\documentclass[11pt,letterpaper]{amsart}

\usepackage{amsmath,amssymb,amsthm}
\usepackage[english]{babel}

\newtheorem{theorem}{Theorem}

\theoremstyle{definition}

\newtheorem{definition}{Definition}[section]
\newtheorem{remark}{Remark}[section]
\newtheorem{lemma}{Lemma}[section]

\title[Topological Expansivity and Shadowing]{Topological Expansivity and Shadowing for Anosov Diffeomorphisms}

\author{Z. Li}

\address{School of Mathematics, Sun Yat-sen University, Zhuhai, China.}
\email{zixu.li915@gmail.com}

\author{C.A. Morales}
\address{Hangzhou International Innovation Institute of Beihang University, Hangzhou 311115, China.}
\email{morales@impa.br}

\author{S. Roma\~na}
\address{School of Mathematics (Zhuhai), Sun Yat-sen University, China}
\email{sergio@mail.sysu.edu.cn}

\author{Y. Yang}

\address{School of Mathematics and Statistics, Liaoning University, Shenyang, 110036, China.}
\email{yangyinong@lnu.edu.cn}

\begin{document}

\begin{abstract}
We show that every Anosov diffeomorphism is topologically expansive in the sense of \cite{lny}. We also prove that Lipschitz shadowing forces the stable and unstable bundles to be uniformly transverse, with a quantitative bound in terms of the shadowing constant. Finally, we give two examples on complete Riemannian manifolds: one is expansive but does not have the shadowing property, while the other has finite volume, bounded sectional curvature, Lipschitz shadowing, and orthogonal invariant bundles, but is not expansive with respect to its Riemannian distance. 
\end{abstract}
\maketitle
\section{Introduction}\label{sec:introduction}

\noindent
For an Anosov diffeomorphism of a compact Riemannian manifold,
hyperbolicity has several familiar uniform consequences. The invariant (i.e. stable and unstable) bundles meet at an angle bounded away from zero \cite{hp}, nearby stable and unstable plaques have a uniform local product structure, the map is expansive, and the shadowing property holds. Compactness makes the local geometric estimates entering these arguments uniform.
On a noncompact manifold the Anosov estimates may remain uniform while the geometric quantities entering the usual compact arguments need not \cite{m}. Local coordinate scales may shrink at infinity, the angle between the invariant bundles may have no positive lower bound, and equivalent Riemannian metrics need not be uniformly equivalent. This makes it natural to ask which expansivity and shadowing conclusions persist without additional uniform geometric control.
Topological formulations of expansivity and shadowing on noncompact spaces were developed by Das, Lee, Richeson and Wiseman
\cite{dlrw}. Lee, Nguyen and Yang subsequently used positive continuous control functions in the study of topological stability and spectral decomposition on noncompact spaces
\cite{lny}. In particular, a continuous expansivity function may tend to zero at infinity while still separating distinct orbits through all iterates.
Under stronger hypotheses one can recover the usual uniform conclusions. Ben Ovadia and DeWitt consider Anosov diffeomorphisms of complete open surfaces under stronger uniform geometric hypotheses \cite{bd}. In that setting a uniform local product structure is available, together with the usual shadowing and expansivity conclusions. For topological expansivity, these additional assumptions are not needed.

Our results describe which parts of the compact Anosov picture survive without uniform control of the geometry at infinity. First, every Anosov-realizable diffeomorphism is topologically expansive. Shadowing behaves differently: we construct an Anosov diffeomorphism that is expansive for its Riemannian distance but does not have the shadowing property. On the other hand, Lipschitz shadowing forces the invariant bundles to be uniformly transverse, with an explicit bound in terms of the shadowing constant. Finally, uniform transversality still does not recover metric expansivity: we construct a complete finite-volume example with bounded sectional curvature, orthogonal invariant bundles, and Lipschitz shadowing that has no expansivity constant for its Riemannian distance.

We now state our results in a precise way.

Let $(Y,\rho)$ be a metric space, endowed with the topology induced by
$\rho$, and let $g\colon Y\to Y$ be a homeomorphism. We say that $g$ is
\emph{expansive} \cite{u} if there exists $\epsilon>0$ such that
\[
\rho\bigl(g^n(x),g^n(y)\bigr)\leq \epsilon
\qquad\text{for every }n\in\mathbb Z
\]
implies that $x=y$. Such a constant is called an \emph{expansivity
constant} for $g$.

If $Y$ is compact, any two metrics inducing its topology are uniformly equivalent. Consequently, if $g$ is expansive with respect to one metric, then it is expansive with respect to every metric. Thus, in the compact case, we may simply say that $g$ is expansive, without specifying the metric.

On a noncompact space, however, expansivity may depend on the choice of metric. To overcome this difficulty, the authors of \cite{lny} introduced the following notion.

\begin{definition}
A homeomorphism $g\colon Y\to Y$ is said to be
\emph{topologically expansive} if there exists a continuous function
$\varepsilon\colon Y\to(0,\infty)$ such that
\[
\rho\bigl(g^n(x),g^n(y)\bigr)
  <\varepsilon\bigl(g^n(x)\bigr)
\qquad\text{for every }n\in\mathbb Z
\]
implies that $x=y$. Such a function is called an
\emph{expansivity function} for $g$.
\end{definition}

Topological expansivity is invariant under topological conjugacy and,
in particular, is independent of the chosen metric; see \cite{y}.

Throughout, $M$ denotes a connected finite-dimensional
differentiable manifold with a fixed differentiable structure and positive dimension.
A $C^1$ diffeomorphism $f\colon M\to M$ is said
to be \emph{Anosov with respect to a Riemannian metric}, whose induced norm is denoted by $\|\cdot\|$, if there exist a continuous $Df$-invariant splitting
\[
TM=E^s\oplus E^u
\]
and constants $C>0$ and $\lambda\in(0,1)$ such that
\[
\left\lVert \left.Df_x^n\right|_{E_x^s}\right\rVert
   \leq C\lambda^n,
\qquad
\left\lVert \left.Df_x^{-n}\right|_{E_x^u}\right\rVert
   \leq C\lambda^n
\]
for every $x\in M$ and every $n\geq 0$.

If $M$ is compact, any two Riemannian metrics on $M$ are uniformly
bi-Lipschitz equivalent. It follows that if $f$ is Anosov with respect
to one Riemannian metric, then it is Anosov with respect to every
Riemannian metric. Thus, in the compact case, we may simply say that
$f$ is Anosov, without specifying the Riemannian metric.

On a noncompact manifold, however, to be Anosov depends on the Riemannian metric, as will be illustrated by the examples
below. This motivates the following terminology.

\begin{definition}
A diffeomorphism $f\colon M\to M$ is called {\em Anosov-realizable} (or
\emph{Anosable} for short) if it is Anosov with respect to some Riemannian metric of $M$.
\end{definition}

It is well known that every Anosable diffeomorphism of a compact manifold is expansive. This conclusion need not hold on a noncompact manifold; the affine example below gives a simple instance. See also \cite{bud} for the corresponding phenomenon in the setting of geodesic flows. Nevertheless, we prove that topological expansivity always holds.

\begin{theorem}\label{thm:topological-expansivity}
Every Anosable diffeomorphism is topologically expansive.
\end{theorem}
The proof combines local cone expansion with a variable neighborhood of the diagonal to construct an expansivity function.

We record two simple consequences of
Theorem~\ref{thm:topological-expansivity}.

\begin{remark}\label{ascenso}
Let $d\geq2$, let $A\in\operatorname{GL}(d,\mathbb R)$, and let
$v\in\mathbb R^d$. The result of \cite{lrr} shows that the affine map
$$
f_{A,v}(x)=Ax+v
$$
admits a complete Riemannian metric with respect to which it is Anosov
whenever either $A$ is hyperbolic or
$$
v\notin\operatorname{Im}(I-A).
$$
Hence Theorem~\ref{thm:topological-expansivity} implies that
$f_{A,v}$ is topologically expansive in either case. For example,
$$
f(x,y)=(x+1,-y)
$$
is topologically expansive, although it is not expansive for the Euclidean metric.
\end{remark}

The converse to Theorem~\ref{thm:topological-expansivity} is false even on compact manifolds: there exist expansive quasi-Anosov diffeomorphisms that are not Anosov \cite{fr}.

Theorem~\ref{thm:topological-expansivity} also gives a simple obstruction to Anosov realizability. A point $x\in Y$ is called \emph{topologically equicontinuous} for $g$ if, for every continuous function $\varepsilon:Y\to(0,\infty)$, there exists a
neighborhood $V$ of $x$ such that
$$
\rho(g^n(x),g^n(y)) < \varepsilon(g^n(x))
\qquad\text{for every }n\in\mathbb Z
$$
whenever $y\in V$.

A topologically expansive homeomorphism of a metric space without isolated points has no topologically equicontinuous points. Indeed,
applying the definition at such a point to an expansivity function would force $V=\{x\}$. Consequently, an Anosable diffeomorphism has no topologically equicontinuous points. In particular, the identity map and Euclidean rotations are not Anosable.

We next turn to shadowing. We say that a bi-infinite sequence $(x_n)_{n\in\mathbb Z}$ is a
$\delta$-pseudo-orbit of $g$ if
$$\rho(g(x_n),x_{n+1})\leq\delta$$
for every $n\in\mathbb Z$. The map $g$ has the {\em shadowing property} (or is {\em shadowable}) if, for every $\varepsilon>0$, there is $\delta>0$ such
that every $\delta$-pseudo-orbit is $\varepsilon$-shadowed by an orbit, that is, there is $x\in Y$ such that $$\rho(g^n(x),x_n)\leq\varepsilon$$
for all $n\in\mathbb{Z}$.

It is also well known that every Anosov diffeomorphism of a compact
manifold has the shadowing property. On a noncompact space, however,
the shadowing property can depend on the chosen compatible metric; see,
for example, \cite{dlrw,dgs}. Cousillas observed that the planar saddle
$$
A(s,u)=\left(\frac{s}{2},2u\right)
$$
can lose the metric shadowing property after a compatible remetrization
\cite{c}. The next result gives the same phenomenon for a complete
conformal Riemannian metric while preserving the usual stable and unstable
directions and hyperbolic rates.

\begin{theorem}\label{ex:no-shadowing}
There exists a complete Riemannian metric on $\mathbb R^2$ under which the linear operator
$$
A(s,u)=\left(\frac{s}{2},2u\right)
$$
is Anosov but not shadowable.
\end{theorem}

We say that $f$ has the Lipschitz shadowing property if there are $L>0$ and $\delta_0>0$ such that, for every $0<\delta<\delta_0$ and every $\delta$-pseudo-orbit $(x_n)_{n\in\mathbb Z}$, there is $z\in M$ satisfying
$$
d(f^n(z),x_n)\leq L\delta
\qquad\text{for every }n\in\mathbb Z.
$$
Since increasing $L$ preserves this property, we shall always take $L\geq1$.

A stronger conclusion follows if shadowing is assumed with a uniform Lipschitz estimate. In the compact setting, Pilyugin and Tikhomirov used bounded solutions of variational equations in their study of Lipschitz
shadowing and structural stability \cite{pt}. This is closely related to the classical connection between bounded solvability, exponential dichotomies, and transversality developed by Pliss \cite{pl} and Palmer
\cite{pal1,pal2}. In the present setting, a direct pointwise argument gives the following quantitative estimate.
\begin{theorem}\label{thm:shadowing-angle}
Let $f:M\to M$ be an Anosov diffeomorphism with the Lipschitz shadowing
property with constant $L\geq 1$. Then its stable and unstable bundles are uniformly
transverse. More precisely, if $\pi_x^s$ and $\pi_x^u$ are the complementary
projections associated with the Anosov splitting, then
$$
\|\pi_x^s\|,\ \|\pi_x^u\|\leq L
$$
for every $x\in M$, and consequently
$$
\sin\angle(E_x^s,E_x^u)\geq\frac{1}{L}.
$$
\end{theorem}


The proof is pointwise and does not require a uniform coordinate radius.

However, uniform transversality does not restore an expansivity constant for the Riemannian distance. A warped product construction gives the following example.

\begin{theorem}\label{ex:warped}
There exists a connected complete Riemannian manifold of finite volume and bounded sectional curvature carrying an Anosov diffeomorphism with the Lipschitz shadowing property, with orthogonal invariant bundles, but which is not expansive with respect to its Riemannian distance.
\end{theorem}

A related loss of expansivity through shrinking metrics occurs for the non-autonomous Anosov family in \cite[Example~2.4]{ma}. In contrast,
Theorem~\ref{ex:warped} gives a single autonomous diffeomorphism on a connected complete manifold, with orthogonal invariant bundles and uniform Lipschitz shadowing.

The construction is a warped product whose torus fibers shrink along suitable escaping orbits. Uniform control of the successive fiber weights is enough to prove Lipschitz shadowing by a weighted Green series.

\section{Topological expansivity}\label{sec:expansivity}

\noindent
Let $TM=E^s\oplus E^u$ be the Anosov splitting. Let
$\pi_x^s:T_xM\to E_x^s$ and $\pi_x^u:T_xM\to E_x^u$ denote the projections
associated with this splitting. There are $C>0$ and $\lambda\in(0,1)$ such
that
$$
\|Df^n|_{E^s_x}\|\leq C\lambda^n,
\qquad
\|Df^{-n}|_{E^u_x}\|\leq C\lambda^n
$$
for every $x\in M$ and $n\geq0$. Choose $N\geq1$ so that
$a:=C\lambda^N<1$, and set $F=f^N$. Then
$$
\|DF_xv^s\|\leq a\|v^s\|,
\qquad
\|DF^{-1}_{F(x)}v^u\|\leq a\|v^u\|
$$
for $v^s\in E^s_x$ and $v^u\in E^u_{F(x)}$. For
$v=v^s+v^u\in T_xM$, define the max splitting norm by
$\|v\|_{*,x}=\max\{\|v^s\|,\|v^u\|\}$. The continuity of the splitting makes
this a continuous norm on $TM$. Define the stable and unstable cones by
$$
C^u_x=\{v=v^s+v^u:\|v^s\|\leq\|v^u\|\},
\qquad
C^s_x=\{v=v^s+v^u:\|v^u\|\leq\|v^s\|\}.
$$
Every nonzero vector belongs to at least one of these cones.

If $v=v^s+v^u\in C^u_x$, then $\|v\|_{*,x}=\|v^u\|$ and
$$
\|DF_xv^s\|\leq a\|v^s\|\leq a\|v^u\|,
\qquad
\|DF_xv^u\|\geq a^{-1}\|v^u\|.
$$
Consequently,
$$
DF_xv\in C^u_{F(x)}
\quad\text{and}\quad
\|DF_xv\|_{*,F(x)}\geq a^{-1}\|v\|_{*,x}.
$$
The same argument for $F^{-1}$ gives
$DF^{-1}_x(C^s_x)\subset C^s_{F^{-1}(x)}$ and
$$
\|DF^{-1}_xv\|_{*,F^{-1}(x)}
\geq a^{-1}\|v\|_{*,x},
\qquad v\in C^s_x.
$$

Choose $1<\mu<a^{-1}$. We transfer these estimates to nearby points. For
$x\in M$, define
$\Phi_x(v)=\exp_{F(x)}^{-1}\bigl(F(\exp_xv)\bigr)$ wherever the exponential
coordinates are defined. Then $\Phi_x(0)=0$ and $D\Phi_x(0)=DF_x$.
Similarly, set
$\Psi_x(v)=\exp_{F^{-1}(x)}^{-1}\bigl(F^{-1}(\exp_xv)\bigr)$, so that
$\Psi_x(0)=0$ and $D\Psi_x(0)=DF^{-1}_x$. Choose $\eta>0$ so that
$a+\eta<a^{-1}-\eta$ and $a^{-1}-\eta>\mu$. All operator norms denoted by
$\|\cdot\|_*$ below are induced by the splitting norms on the corresponding
source and target tangent spaces. By the continuity of $D\Phi_x$ at zero,
for each $x\in M$ and all sufficiently small $v$ we may require
$$
\|D\Phi_x(tv)-D\Phi_x(0)\|_*<\eta,
\qquad 0\leq t\leq1.
$$
Since
$$
\Phi_x(v)-DF_xv
=\int_0^1\bigl(D\Phi_x(tv)-D\Phi_x(0)\bigr)v\,dt,
$$
we obtain
$$
\|\Phi_x(v)-DF_xv\|_{*,F(x)}\leq\eta\|v\|_{*,x}.
$$
If $v\in C^u_x$, it follows that
$$
\|\pi^s_{F(x)}\Phi_x(v)\|\leq(a+\eta)\|v^u\|,
\qquad
\|\pi^u_{F(x)}\Phi_x(v)\|\geq(a^{-1}-\eta)\|v^u\|.
$$
Therefore
$$
\Phi_x(v)\in C^u_{F(x)},
\qquad
\|\Phi_x(v)\|_{*,F(x)}\geq\mu\|v\|_{*,x}.
$$
Applying the same argument to $\Psi_x$ shows that, for sufficiently small
$v\in C^s_x$,
$$
\Psi_x(v)\in C^s_{F^{-1}(x)},
\qquad
\|\Psi_x(v)\|_{*,F^{-1}(x)}\geq\mu\|v\|_{*,x}.
$$

The continuity of $DF$ and $DF^{-1}$, the smooth dependence of the
exponential map, and the continuity of the splitting show that these
strict inequalities persist on an open neighborhood of the zero section
in $TM$. Intersecting this neighborhood with a sufficiently small
tubular neighborhood of the zero section, we may assume that
$$
(x,v)\longmapsto(x,\exp_xv)
$$
is a diffeomorphism onto an open neighborhood $\Omega$ of the diagonal in
$M\times M$. We may also require that, for every $(x,y)\in\Omega$, the
vector $v=\exp_x^{-1}(y)$ is defined and satisfies
$\|v\|_{*,x}<1$.

If $v$ lies in $C^u_x$, the forward estimate for $\Phi_x$
applies; if $v$ lies in $C^s_x$, the backward estimate for $\Psi_x$ applies.

Equip $M\times M$ with the product metric
$D((x,y),(x',y'))=d(x,x')+d(y,y')$, and set
$K=(M\times M)\setminus\Omega$. If $K=\varnothing$, take
$\delta\equiv1$. Otherwise, define
$$
q(x)=\operatorname{dist}_D((x,x),K),
\qquad
\delta(x)=\frac12\min\{1,q(x)\}.
$$
Because $\Omega$ is open and contains the diagonal, $q(x)>0$ for every
$x\in M$. The distance to a closed set is continuous, so
$\delta:M\to(0,\infty)$ is continuous. If $d(x,y)\leq\delta(x)$, then
$$
D\bigl((x,x),(x,y)\bigr)
=d(x,y)
\leq\delta(x)
\leq\frac{q(x)}2
<q(x).
$$
Thus $(x,y)\notin K$, and therefore $(x,y)\in\Omega$. The same conclusion
is immediate when $K=\varnothing$.

Suppose now that $x,y\in M$ satisfy
$$
d(F^n(x),F^n(y))\leq\delta(F^n(x))
\qquad\text{for every }n\in\mathbb Z.
$$
Set $x_n=F^n(x)$, $y_n=F^n(y)$, and
$v_n=\exp_{x_n}^{-1}(y_n)$. Then $\|v_n\|_{*,x_n}<1$ for every
$n\in\mathbb Z$.

Suppose that $x\neq y$. Then $v_0\neq0$, so
$v_0\in C^u_x\cup C^s_x$. If $v_0\in C^u_x$, the definitions give
$$
v_{n+1}
=\exp_{F(x_n)}^{-1}\bigl(F(\exp_{x_n}(v_n))\bigr)
=\Phi_{x_n}(v_n).
$$
Induction with the forward cone estimate yields $v_n\in C^u_{x_n}$ and
$$
\|v_n\|_{*,x_n}\geq\mu^n\|v_0\|_{*,x},
\qquad n\geq0.
$$
This contradicts $\|v_n\|_{*,x_n}<1$. If $v_0\in C^s_x$, then
$v_{n-1}=\Psi_{x_n}(v_n)$. The backward cone estimate gives
$v_{-n}\in C^s_{x_{-n}}$ and
$$
\|v_{-n}\|_{*,x_{-n}}\geq\mu^n\|v_0\|_{*,x},
\qquad n\geq0,
$$
which gives the same contradiction. Therefore $x=y$, and $F$ is topologically expansive.

Finally, suppose that $d(f^n(x),f^n(y))\leq\delta(f^n(x))$ for every
$n\in\mathbb Z$. For every $k\in\mathbb Z$, we then have
$d(F^k(x),F^k(y))=d(f^{Nk}(x),f^{Nk}(y))\leq\delta(F^k(x))$. The
topological expansivity of $F$ gives $x=y$, so $f$ is topologically expansive.
\qed

\section{Failure of shadowing}\label{sec:shadowing}

\noindent
Let $P(s,u)=su$. Choose a smooth function $\rho:\mathbb R\to[0,1]$ such that
$\rho=1$ on $[-2,-1]$ and
$\operatorname{supp}\rho\subset(-5/2,-1/2)$. For $n\geq1$, set
$t_n=2^{-n}$, $R_n=2^{4n}$, $c_n=R_n^2-t_n^2$, and
$H_n=2\sqrt2\,(R_n+1)$. Define
$$
w(p)=1+\sum_{n=1}^{\infty}H_n\rho(p-c_n),
\qquad
g=w(su)^2(ds^2+du^2).
$$
The supports of the summands are pairwise disjoint and tend to infinity, so
the sum is locally finite and $w$ is smooth.

Since $w\geq1$, the metric $g$ dominates the Euclidean metric. A
$g$-Cauchy sequence is therefore Euclidean Cauchy. Near its Euclidean limit,
$w$ is bounded, so Euclidean convergence also implies $g$-convergence. Thus
$g$ is complete.

The function $P$ is invariant under $A$. Consequently, the conformal factor
is constant along every orbit. With $E^s=\mathbb R(1,0)$ and
$E^u=\mathbb R(0,1)$, for every $k\geq0$,
$$
\|DA^kv^s\|_{g,A^kx}=2^{-k}\|v^s\|_{g,x},
\qquad
\|DA^{-k}v^u\|_{g,A^{-k}x}=2^{-k}\|v^u\|_{g,x}.
$$
The bundles are orthogonal because $g$ is conformal to the Euclidean metric.
Hence $A$ is Anosov with $C=1$ and $\lambda=1/2$.

To show that shadowing fails, put
$$
a_n=(R_n,R_n),
\qquad
q_n=(R_n+t_n,R_n-t_n),
\qquad
r_n=(R_n,R_n-t_n).
$$
The segment $\gamma_n(\tau)=(R_n+\tau,R_n-\tau)$,
$0\leq\tau\leq t_n$, joins $a_n$ to $q_n$. Along this segment,
$P(\gamma_n(\tau))=R_n^2-\tau^2\in[c_n,R_n^2]$, so $w=1$ and the
$g$-length of the segment is $\sqrt2\,t_n$. Since $g$ dominates the
Euclidean metric,
\begin{equation}\label{eq:small-jump}
d_g(a_n,q_n)=\sqrt2\,t_n\longrightarrow0.
\end{equation}

On the other hand, $P(r_n)=R_n^2-R_nt_n<c_n-2$, while $w=1+H_n$ on
$I_n=[c_n-2,c_n-1]$. We claim that
\begin{equation}\label{eq:barrier-distance}
d_g(r_n,q_n)\geq1.
\end{equation}
Suppose that a rectifiable curve $\gamma$ from $r_n$ to $q_n$ had $g$-length
less than $1$. Because $g$ dominates the Euclidean metric, the curve would remain in the Euclidean unit ball about $r_n$. Thus, if $(s,u)\in \gamma$, then
\[
|s - R_n| \le 1 \quad \text{and} \quad |u - (R_n - t_n)| \le 1.
\]
Therefore,
\[
s \le R_n+1, \qquad u \le R_n - t_n + 1 \le R_n+1,
\]
and, 

\[
\|\nabla P(s,u)\|
=\sqrt{s^2+u^2}
\leq \sqrt{2}\,(R_n+1)
=:K_n.
\]

The function $P\circ\gamma$ starts below $c_n-2$ and ends at $c_n$.
Hence some subarc of $\gamma$ traverses the interval
$I_n=[c_n-2,c_n-1]$. On the Euclidean unit ball above,
$$
\|\nabla P\|\leq K_n,
$$
so this subarc has Euclidean length at least $1/K_n$. Along it,
$P-c_n\in[-2,-1]$, and therefore $\rho(P-c_n)=1$. Its $g$-length is
thus at least
$$
\frac{1+H_n}{K_n}=2+\frac1{K_n}>2,
$$
contradicting the assumption that $\gamma$ has $g$-length less than $1$.
Hence
$$
d_g(r_n,q_n)\geq1.
$$


For each $n$, define
$$
x_k^{(n)}=
\begin{cases}
A^ka_n,&k<0,\\
A^kq_n,&k\geq0.
\end{cases}
$$
This sequence has only one error, from time $-1$ to time $0$, and its size
tends to zero by \eqref{eq:small-jump}. Suppose that the orbit of
$y=(\alpha,\beta)$ $\frac12$-shadows this sequence. Since $g$ dominates the
Euclidean metric, for $k<0$ we have
$2^{-k}|\alpha-R_n|<1/2$. Letting $k\to-\infty$ gives $\alpha=R_n$.
Similarly, for $k\geq0$,
$2^k|\beta-(R_n-t_n)|<1/2$, so $\beta=R_n-t_n$. Thus $y=r_n$. At time
zero, \eqref{eq:barrier-distance} gives $d_g(y,q_n)\geq1$, a contradiction.

Given $\delta>0$, choose $n$ so large that $\sqrt2\,t_n<\delta$. Then
$(x_k^{(n)})_{k\in\mathbb Z}$ is a $\delta$-pseudo-orbit that cannot be
$\frac12$-shadowed. Therefore $A$ does not have the shadowing property.
\qed

\begin{remark}
The metric above still makes $A$ expansive. Indeed, since
$d_g\geq d_{\mathrm{Eucl}}$, every nonzero displacement under
$$
A^n(s,u)=(2^{-n}s,2^nu)
$$
becomes arbitrarily large in one of the two time directions. Thus
Theorem~\ref{ex:no-shadowing} loses shadowing without losing metric
expansivity; Theorem~\ref{ex:warped} below gives the converse separation.
\end{remark}

\section{Lipschitz shadowing and metric expansivity}

\subsection{Lipschitz shadowing and uniform transversality}
\label{subsec:uniform-transversality}

For complementary subspaces $E$ and $F$ of an inner product space, define
$$
\sin\angle(E,F)=\inf_{\substack{v \in F\\ v \neq 0}} \frac{\operatorname{dist}(v,E)}{\|v\|}
=\inf_{\substack{v\in F\\\|v\|=1}}\operatorname{dist}(v,E).
$$
If either subspace is zero, set their angle equal to $\pi/2$.

\begin{lemma}\label{projection norm}
Let \(E, F \) be nonzero complementary subspaces of a finite-dimensional inner product space and let \(\pi^E\) and \(\pi^F\) be the associated complementary projections. Then
\[
\|\pi^F\|=\|\pi^E\| = \frac{1}{\sin \angle (E,F)}.
\]
\end{lemma}
\begin{proof}
Write \(v = e + f\) with \(e \in E\) and \(f \in F\). Since
$$ \operatorname{dist}(f,E)\leq \|e+f\|=\|v\|, $$
the definition of the angle gives
$$ \|\pi^Fv\|=\|f\| \leq \frac{\|v\|}{\sin\angle(E,F)}. $$
Hence
$$ \|\pi^F\|\leq\frac1{\sin\angle(E,F)}. $$
Conversely, by compactness of the unit sphere in \(F\), choose \(f\in F\) with \(\|f\|=1\) such that
$$ \operatorname{dist}(f,E)=\sin\angle(E,F). $$
Let \(e\in E\) be the orthogonal projection of \(f\) onto \(E\), and set \(v=f-e\). Then
$$ \pi^Fv=f, \qquad \|v\|=\operatorname{dist}(f,E).$$
Thus
$$
\|\pi^F\|=\frac{1}{\sin\angle(E,F)}.
$$
Since the angle between two nonzero subspaces is symmetric, applying the
same argument with $E$ and $F$ interchanged gives
$$
\|\pi^E\|=\|\pi^F\|.
$$

\end{proof}


\begin{lemma}\label{lem:diff_distance}
Let $M$ be a Riemannian manifold, and let $G : M \to M$ be a $C^1$ map. Fix $p \in M$. Suppose $u_j, v_j \in T_pM$ with
$$
u_j \to u, \qquad v_j \to v \qquad (j\to\infty),
$$
and let $t_j > 0$ with $t_j \to 0$. Then
$$
\frac{d\bigl(G(\exp_p(t_j u_j)),\; G(\exp_p(t_j v_j))\bigr)}{t_j}
\longrightarrow \|DG_p(u-v)\|.
$$
\end{lemma}

\begin{proof}
Choose normal coordinates at $p$ and $G(p)$, and write $\widetilde G$
for the coordinate representation of $G$. Then
$$
\widetilde G(t_j u_j)-\widetilde G(t_j v_j)=
t_jD\widetilde G_0(u_j-v_j)+o(t_j).
$$
Since the Riemannian distance agrees with the Euclidean distance to first
order in normal coordinates, and since $u_j\to u$ and $v_j\to v$, it follows
that
$$
\frac{d\bigl(G(\exp_p(t_j u_j)),G(\exp_p(t_j v_j))\bigr)}{t_j}
\longrightarrow \|DG_p(u-v)\|.
$$
\end{proof}


\begin{proof}[Proof of Theorem~\ref{thm:shadowing-angle}]
Fix $p\in M$ and $w\in T_pM\setminus \{0\}$. Let $C>0$ and $\lambda\in(0,1)$ be Anosov constants for $f$. For sufficiently small $t>0$, set $q_t=\exp_p(tw)$ and define
$$
x_n^{(t)}=
\begin{cases}
f^n(p),&n<0,\\
f^n(q_t),&n\geq0.
\end{cases}
$$
The only inexact transition is from time $-1$ to time zero, and its error is
$d(p,q_t)=t\|w\|$. Let $\delta_0$ be a threshold associated with the
Lipschitz shadowing property, and choose $t$ with
$t\|w\|<\delta_0$. Lipschitz shadowing gives $z_t\in M$ such that
$$
d(f^n(z_t),x_n^{(t)})\leq Lt\|w\|
\qquad\text{for every }n\in\mathbb Z.
$$

At time zero,
$d(p,z_t)\leq d(p,q_t)+d(q_t,z_t)\leq(L+1)t\|w\|$. For small $t$,
define $c_t=t^{-1}\exp_p^{-1}(z_t)$. The vectors $c_t$ are bounded. By
finite dimensionality, after passing to a sequence $t_j\to0$ we may assume
that $c_{t_j}\to c\in T_pM$.

Fix $m\geq1$. Since $x_{-m}^{(t_j)}=f^{-m}(p)$ and
$z_{t_j}=\exp_p(t_jc_{t_j})$, the shadowing estimate at time $-m$ is
$$
d\bigl(f^{-m}(\exp_p(t_jc_{t_j})),f^{-m}(p)\bigr)
\leq Lt_j\|w\|.
$$
Apply  Lemma \ref{lem:diff_distance} with  $G=f^{-m}$,
$u_j=c_{t_j}$, and $v_j=0$ to obtain
$\|Df_p^{-m}c\|\leq L\|w\|$. Since $m\geq1$ was arbitrary, this bound holds
for every $m\geq1$. Write $c=c^s+c^u$ according to the Anosov splitting.
Then
$$
\|Df_p^{-m}c\|
\geq\frac{\|c^s\|}{C\lambda^m}-C\lambda^m\|c^u\|.
$$
The right-hand side is unbounded when $c^s\neq0$, so $c\in E_p^u$.

For every fixed $m\geq0$, the shadowing estimate gives
$$
d\bigl(f^m(\exp_p(t_jc_{t_j})),f^m(\exp_p(t_jw))\bigr)
\leq Lt_j\|w\|.
$$
The first-order distance limit with $G=f^m$, $u_j=c_{t_j}$, and $v_j=w$
gives $\|Df_p^m(c-w)\|\leq L\|w\|$. Writing $c-w=b^s+b^u$, the Anosov
estimates yield
$$
\|Df_p^m(c-w)\|
\geq\frac{\|b^u\|}{C\lambda^m}-C\lambda^m\|b^s\|.
$$
Since the left-hand side is bounded for every $m\geq0$, we must have $b^u=0$.
Hence $c-w\in E_p^s$.

Taking \(m=0\) in the preceding estimate gives
$$ \|c-w\|\leq L\|w\|. $$
Since \(c\in E_p^u\) and \(w-c\in E_p^s\), uniqueness of the Anosov splitting gives
$$ c=\pi_p^u w, \qquad w-c=\pi_p^s w. $$

Hence \(\|\pi_p^s w\|\leq L\|w\|,\)  and therefore \(\|\pi_p^s\|\leq L.\)
If both invariant subspaces are nonzero, Lemma~\ref{projection norm} yields
$$ \|\pi_p^u\|=\|\pi_p^s\|\leq L. $$
If one of them is zero, the two projection norms are \(0\) and \(1\), so the same bound holds because \(L\geq1\). Thus
$$ \|\pi_p^s\|,\|\pi_p^u\|\leq L. $$
The angle estimate follows from Lemma~\ref{projection norm} in the nontrivial case and from our convention in the trivial case. Since \(p\) was arbitrary, the result follows.
\end{proof}

\begin{remark}
The coefficient $1$ in Theorem~\ref{thm:shadowing-angle} cannot be
improved uniformly if no hyperbolicity rate is fixed. To see this, fix
two complementary nonzero lines $E,F\subset\mathbb R^2$ making an angle
$\theta\in(0,\pi/2)$, and let $P,Q$ be the corresponding complementary
projections. Put
$$
K=\|P\|=\|Q\|=\frac1{\sin\theta}.
$$
For $0<a<1$, define
$$
B_a=aP+a^{-1}Q.
$$
If $(x_n)_{n\in\mathbb Z}$ is a pseudo-orbit with errors
$$
e_n=x_{n+1}-B_ax_n,
$$
then
$$
z_n=
\sum_{j\geq0}a^jPe_{n-1-j}-\sum_{j\geq0}a^{j+1}Qe_{n+j}
$$
satisfies
$$
z_{n+1}=B_az_n+e_n
$$
and
$$
\sup_n\|z_n\|\leq
K\frac{1+a}{1-a}\sup_n\|e_n\|.
$$
Thus $B_a$ has the Lipschitz shadowing property with constant
$$
K\frac{1+a}{1-a}.
$$
Since
$$
K\frac{1+a}{1-a}\longrightarrow K
\qquad\text{as }a\downarrow0,
$$
the coefficient $1$ in Theorem~\ref{thm:shadowing-angle} is optimal when the hyperbolicity rate is allowed to vary.
\end{remark}

\subsection{A warped product construction}
\label{subsec:no-expansivity-constant}

\noindent
Let
$$ A=
\begin{pmatrix}
2&1\\
1&1
\end{pmatrix}
:\mathbb T^2\to\mathbb T^2.
$$
Its eigenvalues are $\mu$ and $\mu^{-1}$, where
$\mu=(3+\sqrt5)/2>2$. Let $g_0$ be the standard flat metric on
$\mathbb T^2$, let $d_0$ be its distance, and denote the stable and unstable
eigenspaces of $A$ by $E_A^s$ and $E_A^u$. Since $A$ is symmetric, these
eigenspaces are orthogonal with respect to $g_0$.

Since $\mu>2$, choose
$1<\alpha<\log_2\mu$ and define
$$
H(u,v)=(2u,v/2),
\qquad
\omega(u,v)=(1+u^2+v^2)^{-\alpha/2}.
$$
Set $M=\mathbb T^2\times\mathbb R^2$, write $y=(u,v)$, and define
$$
F(x,y)=(Ax,Hy),
\qquad
g=\omega(y)^2g_0+ du^2+dv^2.
$$

The manifold $M$ is connected, and its Riemannian distance satisfies
\begin{equation}\label{eq:base-distance}
d_g((x,y),(x',y'))\geq\|y-y'\|.
\end{equation}
Thus every closed $g$-ball has bounded projection to $\mathbb R^2$. On a
bounded subset of the base, $\omega$ has a positive lower bound, so $g$ is
uniformly equivalent there to the product metric. Since $\mathbb T^2$ is
compact, closed $g$-balls are compact. Thus $(M,g)$ is proper, and hence
complete.

Since the torus factor is two-dimensional, the Riemannian volume
form is
$$
d\operatorname{vol}_g=\omega(y)^2\,d\operatorname{vol}_{g_0}\,du\,dv.
$$
Therefore
$$
\begin{aligned}
\operatorname{Vol}_g(M)
&=
\operatorname{Vol}_{g_0}(\mathbb T^2)
\,2\pi\int_0^\infty r(1+r^2)^{-\alpha}\,dr\\
&=
\frac{\pi\,\operatorname{Vol}_{g_0}(\mathbb T^2)}
{\alpha-1}
<\infty.
\end{aligned}
$$

We next verify that the sectional curvature is uniformly bounded. Put
$$
\phi=\log\omega =-\frac{\alpha}{2}\log(1+\|y\|^2).
$$
Then
$$
\nabla\phi=-\frac{\alpha y}{1+\|y\|^2},
\qquad
\nabla^2\phi=-\frac{\alpha}{1+\|y\|^2}I+
\frac{2\alpha}{(1+\|y\|^2)^2}y\otimes y,
$$
so that
$$
\|\nabla\phi\|\leq\frac{\alpha}{2}.
$$
Moreover,
$$S:=\frac{\nabla^2\omega}{\omega}=
\nabla^2\phi+d\phi\otimes d\phi
$$
has radial and tangential eigenvalues
$$
\frac{\alpha((1+\alpha)r^2-1)}{(1+r^2)^2},
\qquad
-\frac{\alpha}{1+r^2},
\qquad r=\|y\|,
$$
respectively. Since $1<\alpha<\log_2\mu<2$, both eigenvalues have absolute value at
most $\alpha$. Hence
$$\|S\|_{\mathrm{op}}\leq\alpha.$$

The standard curvature formulas for a warped product with flat base and
flat fiber show that, with respect to the orthogonal decomposition
$$\Lambda^2(TM)=\Lambda^2(T\mathbb R^2)\oplus
(T\mathbb R^2\wedge T\mathbb T^2)
\oplus\Lambda^2(T\mathbb T^2),
$$
the curvature operator has blocks
$$0,\qquad -S\otimes I_{T\mathbb T^2},\qquad
-\|\nabla\phi\|^2I_{\Lambda^2(T\mathbb T^2)}.
$$
Consequently,
$$|\sec_g|\leq \max\left\{\|S\|_{\mathrm{op}},
\|\nabla\phi\|^2\right\}
\leq\alpha.
$$

Put $\kappa=2^\alpha$. Since
$\frac14(1+\|y\|^2)\leq1+\|Hy\|^2\leq4(1+\|y\|^2)$, we have
\begin{equation}\label{eq:weight-ratio}
\kappa^{-1}\leq\frac{\omega(Hy)}{\omega(y)}\leq\kappa
\qquad\text{for every }y\in\mathbb R^2.
\end{equation}

The invariant splitting of \(F\) is
$$ E^s=(E_A^s\times\{0\})\oplus\mathbb R\partial_v, \qquad E^u=(E_A^u\times\{0\})\oplus\mathbb R\partial_u, $$
and the two bundles are orthogonal. For \(w\in E_A^s\),
$$ \|DF(w,0)\|_{(Ax,Hy)} = \omega(Hy)\|Aw\|_{g_0} \leq \frac{\kappa}{\mu}\, \omega(y)\|w\|_{g_0}, $$
while
$$ \|DF\,\partial_v\|=\frac12\|\partial_v\|. $$
Since the two stable summands are orthogonal,
$$ \|DF|_{E^s}\| \leq \max\left\{\frac12,\frac{\kappa}{\mu}\right\}. $$
The same argument applied to \(F^{-1}\) on \(E^u\) gives
$$ \|DF^{-1}|_{E^u}\| \leq \max\left\{\frac12,\frac{\kappa}{\mu}\right\}. $$
Since \(\alpha<\log_2\mu\), we have \(\kappa/\mu<1\). Hence \(F\) is Anosov.

We next prove the Lipschitz shadowing property. Since $\|\nabla\log\omega(y)\|\leq\alpha/2$, set $\beta=\alpha/2$. For any
$(x,y),(x',y')\in M$, let $d=d_g((x,y),(x',y'))$. Then
\begin{equation}\label{eq:warped-comparison}
\omega(y')d_0(x,x')\leq e^{\beta d}d.
\end{equation}
Indeed, take an arbitrary rectifiable path $\gamma$ joining the two points
with length $\ell>d$. Every base point of $\gamma$ lies within Euclidean
distance $\ell$ of $y'$, so
$\omega(y(t))\geq e^{-\beta\ell}\omega(y')$. The $g_0$-length of the torus
projection of $\gamma$ is at least $d_0(x,x')$. Hence
$\ell\geq e^{-\beta\ell}\omega(y')d_0(x,x')$. Letting
$\ell\downarrow d$ proves \eqref{eq:warped-comparison}. Since $\gamma$ was
arbitrary, the estimate also covers paths that leave the fiber and pass
through regions where $\omega$ is smaller.

Let $0<\delta<1$ and let $(x_n,y_n)_{n\in\mathbb Z}$ be a
$\delta$-pseudo-orbit of $F$. Write
$\xi_n=y_{n+1}-Hy_n=(\xi_n^u,\xi_n^v)$. By
\eqref{eq:base-distance}, $\|\xi_n\|\leq\delta$. Define the two Green series
$$
h_n^u=-\sum_{j\geq0}2^{-j-1}\xi_{n+j}^u,
\qquad
h_n^v=\sum_{j\geq0}2^{-j}\xi_{n-1-j}^v.
$$
They converge absolutely, satisfy $h_{n+1}=Hh_n+\xi_n$, and obey
$|h_n^u|\leq\delta$ and $|h_n^v|\leq2\delta$. Thus
\begin{equation}\label{eq:base-correction}
\|h_n\|\leq\sqrt5\,\delta.
\end{equation}
It follows that $\zeta_n=y_n-h_n$ satisfies $\zeta_{n+1}=H\zeta_n$.

Choose lifts $\widetilde x_n\in\mathbb R^2$ such that
$\widetilde x_{n+1}=A\widetilde x_n+e_n$ for every $n\in\mathbb Z$, with
$\|e_n\|=d_0(Ax_n,x_{n+1})$. Because $A\mathbb Z^2=\mathbb Z^2$, the
lifts can be chosen consistently in both time directions. A minimizing
representative exists for every torus error, and the vectors $e_n$ are bounded
by the diameter of the flat torus.

Set $W_n=\omega(\zeta_n)$. Applying \eqref{eq:warped-comparison} to the
$n$-th error gives
$\omega(y_{n+1})\|e_n\|\leq e^{\beta\delta}\delta$. By
\eqref{eq:base-correction},
$|\log W_{n+1}-\log\omega(y_{n+1})|\leq\beta\sqrt5\,\delta$. Set
$C_e=e^{(\alpha/2)(\sqrt5+1)}$. Since $\delta<1$,
\begin{equation}\label{eq:weighted-error}
W_{n+1}\|e_n\|\leq C_e\delta.
\end{equation}
Also, $\zeta_{n+1}=H\zeta_n$ and \eqref{eq:weight-ratio} give
$\kappa^{-1}\leq W_{n+1}/W_n\leq\kappa$.

Write $A_s=A|_{E_A^s}$ and $A_u=A|_{E_A^u}$, and let $\pi_s$ and
$\pi_u$ be the corresponding orthogonal projections. Put
$\sigma=\kappa/\mu<1$ and define
$$
z_n^s=\sum_{j\geq0}A_s^j\pi_s e_{n-1-j},
\qquad
z_n^u=-\sum_{j\geq0}A_u^{-j-1}\pi_u e_{n+j},
\qquad
z_n=z_n^s+z_n^u.
$$
Both series converge absolutely because the $e_n$ are uniformly bounded and
$\mu>1$.

For the stable terms, $e_{n-1-j}$ is controlled at weight $W_{n-j}$, and
\eqref{eq:weighted-error} gives
$$
\begin{aligned}
W_n\|A_s^j\pi_s e_{n-1-j}\|
&\leq\mu^{-j}\frac{W_n}{W_{n-j}}
 W_{n-j}\|e_{n-1-j}\|\\
&\leq C_e\delta\left(\frac{\kappa}{\mu}\right)^j.
\end{aligned}
$$
For the unstable terms, $e_{n+j}$ is controlled at weight $W_{n+j+1}$, so
$$
\begin{aligned}
W_n\|A_u^{-j-1}\pi_u e_{n+j}\|
&\leq\mu^{-j-1}\frac{W_n}{W_{n+j+1}}
 W_{n+j+1}\|e_{n+j}\|\\
&\leq C_e\delta\left(\frac{\kappa}{\mu}\right)^{j+1}.
\end{aligned}
$$
Summing the geometric series yields
\begin{equation}\label{eq:green-bounds}
W_n\|z_n^s\|\leq\frac{C_e}{1-\sigma}\delta,
\qquad
W_n\|z_n^u\|\leq\frac{C_e\sigma}{1-\sigma}\delta.
\end{equation}

Shifting the two series gives the recurrence $z_{n+1}=Az_n+e_n$. Set
$\widetilde z=\widetilde x_0-z_0$ and let $z\in\mathbb T^2$ be its
projection. Comparing the recurrences in both time directions gives
$\widetilde x_n-A^n\widetilde z=z_n$ for every $n\in\mathbb Z$. By
\eqref{eq:green-bounds},
$W_nd_0(x_n,A^nz)\leq C_e(1+\sigma)\delta/(1-\sigma)$.
Join $(x_n,y_n)$ first to $(x_n,\zeta_n)$ in the Euclidean base and then to
$(A^nz,\zeta_n)$ in the torus fiber. This gives
$$
d_g\bigl((x_n,y_n),F^n(z,\zeta_0)\bigr)\leq L\delta
\qquad\text{for every }n\in\mathbb Z,
$$
where $L=\sqrt5+C_e(1+\sigma)/(1-\sigma)$. Thus $F$ has the Lipschitz
shadowing property with threshold $\delta_0=1$.

It remains to show that no expansivity constant exists. Let
$D=\operatorname{diam}_{d_0}(\mathbb T^2)$, choose distinct
$x,x'\in\mathbb T^2$, and, for $T>0$, put
$p_T=(x,T,T)$ and $q_T=(x',T,T)$. For every $n\in\mathbb Z$,
$H^n(T,T)=(2^nT,2^{-n}T)$, and therefore
$$
\|H^n(T,T)\|^2=T^2(4^n+4^{-n})\geq2T^2
\quad\mbox{ and }\quad
d_g(F^np_T,F^nq_T)
\leq D(1+2T^2)^{-\alpha/2}
$$
for every $n\in\mathbb Z$.
For any $c_0>0$, a sufficiently large $T$ makes the right-hand side smaller
than $c_0$, although $p_T\neq q_T$. Hence there is no $c_0>0$ such that $d_g(F^np,F^nq)<c_0$ for every $n\in\mathbb Z$ forces $p=q$.

\qed

\begin{remark}
The injectivity radius of $g$ is not bounded away from zero. If $\ell_0$ is the length of a shortest nonzero lattice vector of $(\mathbb T^2,g_0)$, then the corresponding deck displacement in the torus fiber has length at most $\omega(y)\ell_0$. Hence
$$
\operatorname{inj}_g(x,y) \leq \frac{\ell_0}{2}\omega(y) \longrightarrow0 \qquad\text{as }\|y\|\to\infty.
$$
Thus the example has bounded sectional curvature but no positive uniform injectivity radius.
\end{remark}


\section*{Funding}

\noindent
YY was partially supported by the National Natural Science Foundation of China (Grant No. 12101281) and Scientific Research Foundation of Education Department of Liaoning Province, China (Grant No. JYTQN2023190).

\section*{Declaration of competing interest}

\noindent
There is no competing interest.

\section*{Data availability}

\noindent
No data was used for the research described in the article.

\end{document}